\documentclass[11pt]{amsart}

\usepackage{pstricks-add}

\usepackage[dvips]{graphicx}
\usepackage{calc}
\usepackage{color}
\usepackage{amsmath}
\usepackage{amssymb}
\usepackage{amscd}
\usepackage{amsthm}
\usepackage{amsbsy}
\usepackage{delarray}
\usepackage{enumerate}
\usepackage[T1]{fontenc}
\usepackage{inputenc}
\usepackage{enumerate}
\usepackage{hyperref}
\usepackage[margin= 1.5 in]{geometry}

\begin{document}

\setlength{\parindent}{5mm}
\renewcommand{\leq}{\leqslant}
\renewcommand{\geq}{\geqslant}
\newcommand{\N}{\mathbb{N}}

\newcommand{\CZ}{\mu_\mathrm{CZ}}
\newcommand{\Z}{\mathbb{Z}}
\newcommand{\R}{\mathbb{R}}
\newcommand{\C}{\mathbb{C}}
\newcommand{\F}{\mathbb{F}}
\newcommand{\g}{\mathfrak{g}}
\newcommand{\h}{\mathfrak{h}}
\newcommand{\K}{\mathbb{K}}
\newcommand{\RN}{\mathbb{R}^{2n}}
\newcommand{\ci}{c^{\infty}}
\newcommand{\derive}[2]{\frac{\partial{#1}}{\partial{#2}}}
\renewcommand{\H}{\mathbb{H}}
\newcommand{\eps}{\varepsilon}
\renewcommand{\r} {\textbf{r}}
\renewcommand{\L}{\mathcal{L}}

\newcommand{\B}{B^{2n}(\frac{1}{\sqrt{\pi}})}

\theoremstyle{plain}
\newtheorem{theo}{Theorem}
\newtheorem{prop}[theo]{Proposition}
\newtheorem{lemma}[theo]{Lemma}
\newtheorem{definition}[theo]{Definition}
\newtheorem*{notation*}{Notation}
\newtheorem*{notations*}{Notations}
\newtheorem{corol}[theo]{Corollary}
\newtheorem{conj}[theo]{Conjecture}
\newtheorem{question}[theo]{Question}
\newtheorem*{question*}{Question}
\newtheorem*{remark*}{Remark}

\newenvironment{demo}[1][]{\addvspace{8mm} \emph{Proof #1.
    ---~~}}{~~~$\Box$\bigskip}

\newlength{\espaceavantspecialthm}
\newlength{\espaceapresspecialthm}
\setlength{\espaceavantspecialthm}{\topsep} \setlength{\espaceapresspecialthm}{\topsep}

\newenvironment{example}[1][]{\refstepcounter{theo} 
\vskip \espaceavantspecialthm \noindent \textsc{Example~\thetheo
#1.} }%
{\vskip \espaceapresspecialthm}

\newenvironment{remark}[1][]{\refstepcounter{theo} 
\vskip \espaceavantspecialthm \noindent \textsc{Remark~\thetheo
#1.} }%
{\vskip \espaceapresspecialthm}

\def\Homeo{\mathrm{Homeo}}
\def\Hameo{\mathrm{Hameo}}
\def\Diffeo{\mathrm{Diffeo}}
\def\Symp{\mathrm{Symp}}
\def\Id{\mathrm{Id}}
\newcommand{\norm}[1]{||#1||}
\def\Ham{\mathrm{Ham}}
\def\Hamtilde{\widetilde{\mathrm{Ham}}}
\def\Crit{\mathrm{Crit}}
\def\Spec{\mathrm{Spec}}
\def\osc{\mathrm{osc}}
\def\Cal{\mathrm{Cal}}

\def\n {\nu_c}

\title[Poisson Bracket invariants]{Spectral killers and Poisson bracket invariants} 
\author{Sobhan Seyfaddini}
\date{\today}

\address{Sobhan Seyfaddini, D\'epartement de Math\'ematiques et Applications de l'\'Ecole Normale Sup\'erieure, 45 rue d'Ulm, F 75230 Paris cedex 05}
\email{sobhan.seyfaddini@ens.fr}


\begin{abstract} 
We find optimal upper bounds for spectral invariants of a Hamiltonian whose support is contained in a union of mutually disjoint displaceable balls.  This gives a partial answer to a question posed by Leonid Polterovich in connection with his recent work on Poisson bracket invariants of coverings.
\end{abstract}

\maketitle

\section{Introduction}
   The theory of spectral invariants associates to each Hamiltonian $H$, on a closed and connected symplectic manifold $(M, \omega)$,  a collection of real numbers
 $$\{c(a, H) \in \mathbb{R}: a \in QH_*(M) \setminus \{0\} \},$$ where $QH_*(M)$ denotes the quantum homology of $M$.  These numbers are referred to as the spectral invariants of $H$.
They were introduced by Oh, Schwarz and Viterbo \cite{Vit92, Sc00, Oh05a}.   Roughly speaking,  $c(a, H)$ is the action level at which $a \in QH_*(M)\setminus \{0\}$ appears in the Floer homology of $H$.  These invariants have been studied extensively and have had many interesting applications in symplectic geometry; see \cite{EP03, Gin, Oh05a, Sc00}. 

In this article we will be only concerned with the spectral invariant associated to the neutral element $[M] \in QH_*(M)$.  Hence, we will abbreviate $c(H):= c([M], H)$.  The main objective of this paper is to find optimal upper bounds for $c(H)$ when the support of $H$ is contained in a disjoint union $U_1 \sqcup \cdots \sqcup U_N$, where each $U_i$ is a displaceable open ball.   A priori, the expected upper bound, given by the triangle inequality (see Proposition  \ref{properties_spec}), for $c(H)$ is  $\sum c(H_i)$, where $H_i := H|_{U_i}$.   From the viewpoint of Morse-Floer theory for the action functional, there exists little communication between Hamiltonians which are supported in small and pairwise disjoint balls. This lack of Floer theoretic interaction among the $H_i$'s is manifested in Theorem \ref{theo_bound}, which provides the same upper bound for $c(H)$ as $c(H_i)$. Thus, the bound in Theorem \ref{theo_bound} is roughly $N$ times  better than the expected upper bound.

To prove Theorem \ref{theo_bound}, we introduce a new technique for bounding $c(H)$  which involves computing spectral invariants for a special class of functions. Because of a fascinating property of this class of functions, we refer to them as ``spectral killers.''  See Theorem \ref{theo_killer}.

The motivation for searching for upper bounds as described above arises from the recent work of Polterovich \cite{Po12b} on Poisson bracket invariants of coverings.  Indeed, as demonstrated in \cite{Po12b}, such upper bounds lead to optimal lower bounds for Poisson bracket invariants; see Theorem \ref{theo_1}. 

\subsection{Upper bounds for spectral invariants and spectral killers}
Recall that a time dependent Hamiltonian $H \in C^{\infty}([0,1] \times M)$ gives rise to a Hamiltonian flow  $\phi^t_H$.  The group of Hamiltonian diffeomorphisms $Ham(M)$ is the collection of time-1 maps of such flows.  The Hofer norm of $\psi \in Ham(M)$ is defined by the expression $\Vert \psi \Vert_{Hofer} = \inf \{ \Vert H \Vert_{(1,\infty)}: \psi = \phi^1_H \},$
	where $$\displaystyle \Vert H \Vert_{(1, \infty)} = \int_0^1{ (\max_M H(t, \cdot) - \min_M H(t, \cdot)) \,dt}.$$

A subset $U \subset M$ is said to be displaceable if there exists a Hamiltonian diffeomorphism $\phi \in Ham(M)$ such that $\phi(U) \cap \bar{U}  = \emptyset$, where $\bar{U}$ denotes the closure of $U$.  Define the displacement energy of $U$ to be $$E(U) = \inf \{\Vert \phi \Vert_{Hofer}: \phi(U) \cap \bar{U} = \emptyset \}.$$
One version of the famous energy-capacity inequality (see \cite{Vit92, usher})  states that if $supp(H)$, the support of $H$, is displaceable, then
$$|c(H)| \leq E(supp(H)).$$

Suppose that $\displaystyle supp(H) \subset U_1 \sqcup \cdots \sqcup U_k $ where the sets $U_i$ are mutually disjoint and each $U_i$ is displaceable; note that $supp(H)$ is not necessarily displaceable. Using the triangle inequality \ref{properties_spec} and the above energy-capacity inequality one can easily show that $|c(H)| \leq \sum_{i=1}^k{E(U_i)}.$
Polterovich, through his work on Poisson bracket invariants \cite{Po12b}, particularly in connection with Question \ref{Leonid_que2}, was lead to ask: 

\begin{question} \label{Leonid_que1}
 Is it true that $|c(H)| \leq \max\{E(U_i)\}$?
\end{question}

In this article, we address the above question on monotone manifolds: We call $(M, \omega)$ monotone if $\exists \, \lambda \neq 0$ such that $\omega|_{\pi_2} = \lambda c_1|_{\pi_2}$, where $c_1$ denotes the first Chern class of $M$.  Note that we allow the monotonicity constant $\lambda$ to be negative.  Monotone manifolds constitute a large and important class of symplectic manifolds; examples include projective spaces.  

Our main theorem gives an affirmative answer to Polterovich's question under certain ``regularity'' assumptions on the sets $U_i$. Below, we assume that $(M, \omega)$ is monotone with monotonicity constant $\lambda$.

\begin{theo}\label{theo_bound}
Let $U_1, \cdots , U_k$ denote a collection of mutually disjoint open subsets of $M$.  Assume that each $U_i$ is symplectomorphic to the Euclidean ball of radius $r_i$ and that the $U_i$'s are displaceable with displacement energy $E(U_i) < \frac{|\lambda|}{2}$.  Then, for any Hamiltonian $H$ whose support is contained in $U_1 \sqcup \cdots \sqcup U_k,$ 
$$0 \leq c(H) \leq \pi r^2,$$
where $r = \max\{r_1, \cdots, r_k\}$.
\end{theo}


\begin{remark}
It follows from Theorem 1.1 of \cite{usher} that the displacement energy of a symplectic ball of radius $r$ is at least as large as $\pi r^2.$ Therefore, it follows that $|c(H)| \leq \max\{E(U_i)\}.$ 
\end{remark}

\begin{remark}\label{rem:upperbound_optimal}
In order to show that the upper bound in the above theorem is optimal, we will now briefly describe how one may construct $H$ such that $c(H) \approx \pi r^2$.  Identify $U_1$ with $B_r,$ the Euclidean ball of radius $r$.  Take $H$ to be a radial Hamiltonian supported in $U_1$ which has no non-constant periodic orbits of period at most $1$.  Proposition 4.1 of \cite{usher} implies that $c(H) = max(H)$. (The conventions in \cite{usher} are different than ours and hence one must adjust the statement in \cite{usher}.)  Now, a simple exercise in calculus would yield that $H$ can be picked such that $max(H)$ is arbitrarily close to (but always smaller than) $\pi r^2$.
\end{remark}

\begin{remark}
It is absolutely crucial to assume the $U_i$'s are displaceable.  Indeed, any disc in $S^2$ with area larger than half the total area supports Hamiltonians  with arbitrarily large spectral invariants.\end{remark}

\begin{remark}\label{rem:conj_max_formula}
The sets $U_i$ in Theorem \ref{theo_bound} and Question \ref{Leonid_que1} are non-overlapping and symplectically small, in the sense that they are displaceable.  As mentioned earlier, this leads to a lack of Floer theoretic interaction among the Hamiltonians $H_i$.  Hence, it seems reasonable to  conjecture that the following maximum formula holds: $$c(H) = \max\{c(H_i)\}.$$  Of course, this would generalize Theorem \ref{theo_bound} and would give an affirmative answer to Question \ref{Leonid_que1}.  Furthermore, Theorem \ref{theo_1} which is a corollary of Theorem \ref{theo_bound} would follow as well.
\end{remark}

\subsection*{Spectral Killers}
To prove Theorem \ref{theo_bound}, we compute the spectral invariant of a special class of functions; we have come to call these functions ``spectral killers.''  Let $U \subset M$ denote an open subset which is symplectomorphic to a ball.  We identify $U$ with $B_r$, the open Euclidean ball of radius $r$.  

\medskip

\noindent Pick $\epsilon \in  (0, \frac{r}{4}).$  Define $K_{\epsilon} : M \rightarrow \mathbb{R}$ such that:

\begin{figure}
\centering
\includegraphics{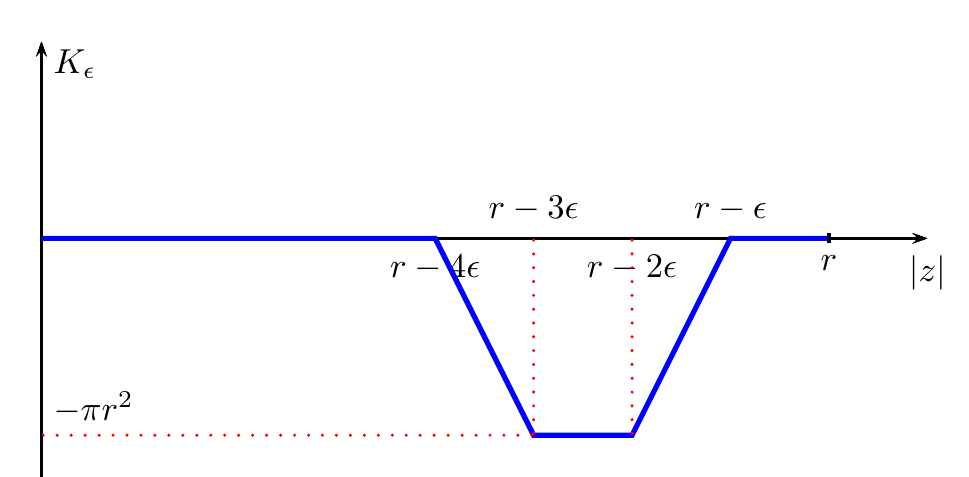}
\caption{Graph of the spectral killer $K_{\epsilon}$.}
\label{fig:killer_graph}
\end{figure}

\begin{enumerate}
\item $supp(K_{\epsilon}) \subset \{z : r - 4 \epsilon \leq |z| \leq r - \epsilon \},$ 
\item $K_{\epsilon}$ is radial, i.e. $K_{\epsilon}(z) = k_\epsilon(|z|)$ is a function of $|z|$,
\item $k_{\epsilon}$ decreases linearly on $[r -4 \epsilon, r - 3 \epsilon],$
\item $K_{\epsilon}(z) = -\pi r^2$ on $\{z : r - 3 \epsilon \leq|z| \leq r - 2 \epsilon \},$
\item $k_{\epsilon}$ increases linearly on $[r -2 \epsilon, r - \epsilon].$
\end{enumerate}

\noindent See Figure \ref{fig:killer_graph} for a graph of $K_{\epsilon}$.  Observe that $K_{\epsilon}$ approximates  the indicator function of the shell $\{z : r - 3 \epsilon \leq|z| \leq r - 2 \epsilon \}$.
Our next result states that adding $K_{\epsilon}$ to a Hamiltonian $H$ supported in $U$ kills the spectral invariants of $H$; hence we call $K_{\epsilon}$  a ``spectral killer'' for the domain $U$.  (Recall that spectral invariants are defined for continuous functions; see Section \ref{sec:sepctral_invariants}.)  

In the following theorem,  $(M, \omega)$ is assumed to be monotone, with monotonicity constant $\lambda$, and $U$, $K_{\epsilon}$ are as described above.

\begin{theo}\label{theo_killer}
  Suppose that $U$ is displaceable with displacement energy $E(U) < \frac{|\lambda|}{2}.$ Then, $c(H+K_{\epsilon}) = 0$  for any Hamiltonian $H$ supported in  $B_{r - 4 \epsilon}$.
\end{theo}

We will see in Section \ref{sec:proof_killer} that Theorem \ref{theo_bound} follows, without much difficulty, from the above theorem.  The important role that spectral killers play in estimating spectral invariants suggests that it would be worthwhile to construct spectral killers for domains more general than symplectic balls. 

\subsection{Lower bounds for Poisson bracket invariants}
   A partition of unity on $M$ is a collection of non-negative smooth functions $\vec{f} = \{f_1, \cdots, f_L\}$ such that $\sum{f_i} =1$. When $M$ is symplectic, the space of smooth functions $C^{\infty}(M)$ can be equipped with the Poisson bracket $\{\cdot, \cdot \} : C^{\infty}(M) \times C^{\infty}(M) \rightarrow C^{\infty}(M)$ which in local Darboux coordinates $(x_i,y_i)$ is given by the expression $ \displaystyle \{f,g\} = \sum_{i} \frac{\partial f} {\partial x_i} \frac{\partial g} {\partial y_i} - \frac{\partial f} {\partial y_i} \frac{\partial g} {\partial x_i} .$
	
	Following \cite{Po12a}, we define the magnitude of Poisson non-commutativity of $\vec{f}$
	$$\n(\vec{f}) : = \max_{x_i, y_i \in [-1,1]} \Vert \{ \sum{x_i f_i}, \sum{y_i f_i} \} \Vert,$$

\noindent where  $\Vert \cdot \Vert$ stands for the sup norm.

	Let $\mathcal{U} = \{U_1, \cdots , U_N \}$ denote a finite open cover of $M$.  The partition of unity $\{f_i\}$ is said to be subordinate to $\mathcal{U}$ if the support of each $f_i$ is contained in one of the  $U_j$'s.  In \cite{Po12a}, Polterovich defines $pb(\mathcal{U})$ the Poisson bracket invariant of  $\mathcal{U}$ by
	$$pb(\mathcal{U}) = \inf \n (\vec{f}),$$
\noindent where the infimum is taken over all partitions of unity subordinate to $\mathcal{U}$.  This invariant provides an obstruction to existence of a Poisson commuting partition of unity subordinate to $\mathcal{U}$.  

The main application of Theorem \ref{theo_bound} is in providing lower bounds for the invariant $pb$.   The search for lower bounds for certain invariants related to $pb$ was initiated in \cite{EPZ07}. Lower bounds for $pb$ and their connections to quantum mechanics are studied extensively in \cite{Po12a, Po12b}.   As demonstrated in \cite{Po12a, Po12b}, on certain quantizable symplectic manifolds,  the quantity $\nu_c$ admits a quantum mechanical counter part $\kappa_q$.  Lower bounds for $pb(\mathcal{U})$ provide lower bounds for  $\kappa_q$ and the \emph{inherent noise} of  \emph{quantum registration procedures} which arise.

If a cover $\mathcal{U}$ consists of non-displaceable sets, then $pb(\mathcal{U})$ could vanish; see any of \cite{EPZ07, Po12a, Po12b} for such examples. Suppose that $\mathcal{U}  = \{U_1, \cdots, U_N\}$ consists of displaceable open sets and let $E(\mathcal{U}) = \max\{E(U_i)\}$. In \cite{Po12b}, Question 8.1, Polterovich asks:

\begin{question} \label{Leonid_que2}
Is it true that $pb(\mathcal{U}) \geq \frac{C}{E(\mathcal{U})}$, where the constant $C$ depends only on $(M, \omega)$?
\end{question}

This question remains open.  In all  currently known lower bounds, the constant $C$ depends on $\mathcal{U}.$ In \cite{Po12b}, the above question is studied for covers which satisfy certain regularity conditions.  One of these assumptions is $d$-regularity:  

\medskip

\noindent \textbf{A1.} A covering $\mathcal{U}$ is said to be $d$-regular if the closure of every $U_j$ intersects at most the closure of $d$ other sets from the cover. 

\medskip

 The following result is the main application of Theorem \ref{theo_bound}.  The fact that a result of this nature would follow from Theorem \ref{theo_bound} was explained to us by Leonid Polterovich. 

\begin{theo} \label{theo_1} Suppose that $M$ is monotone with monotonicity constant $\lambda$, $\mathcal{U}$ is $d$-regular, and each $U_i$ is symplectomorphic to the Euclidean ball of radius $r_i$. If $E(\mathcal{U}) < \frac{|\lambda|}{2}$, then
$$pb(\mathcal{U}) \geq \frac{C(d)}{\pi r^2},$$
where $r = \max r_i$ and $C(d) = \frac{1}{2 d^2}.$
\end{theo}

In Theorem 4.8 of \cite{Po12b}, the estimate appearing in Theorem \ref{theo_1} is proven for more general coverings but under the assumption that $\omega|_{\pi_2(M)} = 0$; see Remark \ref{rem:aspherical_case} for a related discussion.  

To handle more general symplectic manifolds, Polterovich introduces the notion of $p$-regularity:  For a subset $Z \subset M$ define its star $St(Z) = \cup U_i$ where the union is taken over all $U_i$'s such that $\bar{Z} \cap \bar{U_i} \neq \emptyset.$

\noindent \textbf{A2.}  A covering $\mathcal{U}$ is said to be $p$-regular if for every $U_i$ there exists a Hamiltonian $F_i$ supported in the $p$-times iterated star of $U_i$, $St( \cdots (St(U_i)) \cdots)$ such that $\phi^1_{F_i}(U_i) \cap \bar{U_i} =\emptyset$.

\medskip

A covering is said to be $(d,p)$-regular if it is both $d$-regular and $p$-regular. It is proven in \cite{Po12b} that if $\mathcal{U}$ is $(d,p)$-regular then
$$pb(\mathcal{U}) \geq \frac{C(d,p)}{E(\mathcal{U})},$$
where $C(d, p)$ depends only on $d, p$.

The significance of Theorem \ref{theo_1} is that it allows us to remove the $p$-regularity condition on monotone manifolds, albeit for coverings which consist of symplectic balls. Intuitively, we expect $pb(\mathcal{U})$ to  be larger for more irregular covers and so it is expected that both $d$-regularity and $p$-regularity can be removed.  Theorem \ref{theo_1} is a first step towards this goal.

\begin{remark}
It is not explicitly stated in \cite{Po12b} that $C(d, p) = \frac{1}{2 (d^{2p} + 1)^2}.$ However, this can be extracted from Proposition 4.6, Theorem 4.8 and Corollary 5.3.
\end{remark}  

\begin{remark} \label{rem:aspherical_case}
A symplectic manifold $(M, \omega)$ is said to be aspherical if $\omega|_{\pi_2} = 0$.  Stronger versions of Theorems \ref{theo_bound}, \ref{theo_killer} \& \ref{theo_1} hold on such manifolds. The assumptions in the statements of these theorems relating to the monotonicity constant $\lambda$ become unnecessary.  Furthermore,  the assumptions regarding the displaceablity of $U$ in Theorem \ref{theo_killer} and $U_i$'s in Theorems  \ref{theo_bound} and \ref{theo_1}  can be entirely eliminated.  We will not prove any of the above statements in this article.  Their proofs are similar to, and in fact easier than, the proofs presented here.

With regards to  coverings on aspherical manifolds, our techniques do not allow us to obtain results as general as Theorem 4.8(ii) and Proposition 5.4 of \cite{Po12b}.
\end{remark}

\subsection*{Acknowledgments}  The fact that affirmative answers to Questions \ref{Leonid_que1}, \ref{Leonid_que2} would imply Theorem \ref{theo_1} was explained to me by  Leonid Polterovich during my visit to the University of Chicago in May of 2012.  I am very thankful to him for having invited me to Chicago, for suggesting these questions to me, and for many helpful conversations about the contents of this paper.  I am also very grateful to Claude Viterbo for several helpful discussions about computing the spectral invariants of spectral killers.  I would also like to thank Lev Buhovsky, Vincent Humili\`ere, R\'emi Leclercq and Fr\'ed\'eric Le Roux for interesting conversations relating to this paper.  Lastly, I would like to thank an anonymous referee for his/her valuable comments.

The research leading to these results has received funding from the European Research Council under the European Union's Seventh Framework Programme (FP/2007-2013) / ERC Grant Agreement  307062.

\section{Preliminaries on Spectral Invariants} \label{sec:sepctral_invariants}
In this section, we recall the aspects of the theory of spectral invariants required to prove Theorems \ref{theo_bound} and \ref{theo_killer}; for more details please see \cite{mcduff-salamon, Oh05b}.  Throughout this section, we suppose that $(M, \omega)$ is a closed, connected and monotone symplectic manifold of dimension $2n$.

\subsection*{The action functional and the spectrum}
 We denote by $\Omega_0(M)$ the space of contractible loops in $M$.  Define $\Gamma:= \frac{\pi_2(M)}{\ker(c_1)} = \frac{\pi_2(M)}{ \ker([\omega])}$; this is the group of deck transformations of the Novikov covering of $\Omega_0(M)$,  defined by
 $$\tilde{\Omega}_0(M) = \frac{ \{ [z,u]: z \in \Omega_0(M) , u: D^2 \rightarrow M , u|_{\partial D^2} = z \}}{[z,u] = [z', u'] \text { if } z=z' \text{ and } \bar{u} \# u' = 0 \text{ in } \Gamma},$$
   where $\bar{u} \# u'$ is the sphere obtained by gluing $u$, with its orientation reversed, to $u'$ along their common boundary. The disc $u$ in $[z, u]$, is referred to as the capping disc of the orbit $z$.  Recall that the action functional $\mathcal{A}_H: \tilde{\Omega}_0(M) \rightarrow \mathbb{R}$, associated to a Hamiltonian $H$, is defined by
   $$\mathcal{A}_H([z,u]) =  \int_{0}^{1} H(t,z(t))dt \text{ }- \int_{D^2} u^*\omega.$$
The set of critical points of $\mathcal{A}_H$, denoted by $\Crit(\mathcal{A}_H)$, consists of equivalence classes,  $[z,u] \in \tilde{\Omega}_0(M)$, such that $z$ is a $1$--periodic orbit of the Hamiltonian flow $\phi^t_H$.

It is well known that $Crit(\mathcal{A}_H) = \{ [z,u] : \text{ $z$ is a 1-periodic orbit of } \phi^t_H \}$ is the set of critical points of $\mathcal{A}_H$.  
The action spectrum of $H$, denoted by $Spec(H)$, is the set of critical values of $\mathcal{A}_H$; it has Lebesgue measure zero.  

\subsection*{The Conley--Zehnder index} When $H$ is non-degenerate $\Crit(\mathcal{A}_H)$ can be indexed by the well known Conley--Zehnder index  $\CZ: \Crit(\mathcal{A}_H) \rightarrow \mathbb{Z}$. Here, we will recall some facts about $\CZ$ without defining it.

Many conventions are used for normalizing $\CZ$.  Our convention is as follows: Suppose that $g$ is a $C^2$--small Morse function. We normalize the Conley--Zehnder index so that for every critical point $p$ of $g$, 
$$ \mu_\mathrm{CZ}([p, u_p]) = i_{\text{Morse}}(p) - n,$$
where $i_{\text{Morse}}(p)$ is the Morse index of $p$ and $u_p$ is a trivial capping disc.  For every $A \in \Gamma$, the Conley--Zehnder index satisfies the following identity
\begin{equation}\label{CZ-index identity}
   \CZ([z,u\#A]) = \CZ([z,u]) - 2 c_1(A).
\end{equation}

\subsection*{Spectral Invariants} Spectral invariants are defined for non-degenerate Hamiltonians via Hamiltonian Floer theory.  The Floer homology $HF_*(H)$ is filtered by values of the action functional.  One associates to a quantum homology class, viewed as a Floer homology class via the PSS isomorphism $\Phi_{PSS} : QH_*(M) \to HF_{*-n}(H)$ \cite{PSS}, the minimal action level at which it appears in $HF_*(H)$. The specific spectral invariant used in this article is the one associated to $[M] \in QH_*(M)$; we will denote it by $c(H)$ for $H \in C^{\infty}([0,1] \times M)$.

Well-known estimates for actions of Floer trajectories imply that $|c(H) - c(G)| \leq \int_{0}^{1} \max_{x\in M} |H_t -G_t|  dt $.  This inequality allows us to  define $c(H)$ for continuous Hamiltonians (such as spectral killers $K_\epsilon$): we set $c(H) = \lim c(H_i)$ where $H_i$ is a sequence of smooth and non-degenerate Hamiltonians converging uniformly to $H$.

We will now list, without proof, those properties of $c: C^{\infty}([0,1] \times M)  \rightarrow \mathbb{R}$ which will be use later on; see \cite{Oh05a, Oh05b, Sc00, usher} for further details. Recall that the composition of two Hamiltonian flows, $\phi^t_H \circ \phi^t_G$, and the inverse of a flow, $(\phi^t_H)^{-1},$ are Hamiltonian flows generated by $H\#G(t,x) = H(t,x) + G(t, (\phi^t_H)^{-1}(x))$ and $\bar{H}(t,x) = -H(t, \phi^t_H(x))$, respectively.
\begin{prop} \label{properties_spec}
    Let $(M, \omega)$ denote a monotone symplectic manifold of dimension $2n$.  The spectral invariant $c: C^{\infty}([0,1] \times M)  \rightarrow \mathbb{R}$ has the following properties:
    \begin{enumerate}   
   \item \label{normalization}(Normalization) $c(0) =0$.
    \item \label{monotonicity} (Monotonicity) If $H \leq G$ then $c(H) \leq c(G)$.
    \item \label{triangle} (Triangle Inequality) $c(H\#G) \leq c(H) + c(G)$.
    \item \label{continuity}(Continuity)  $|c(H) - c(G)| \leq \Vert H - G \Vert_{\infty}.$
    \item \label{spectrality} (Spectrality) $c(H) \in Spec(H)$, i.e. there exists $[z,u] \in \Crit(\mathcal{A}_H)$ such that $c(H) = \mathcal{A}_H([z,u])$. Moreover, if $H$ is non-degenerate then $\CZ([z,u]) = n$.  
    \end{enumerate}   
   \end{prop} 
	

\section{Proofs of Theorems \ref{theo_bound} \& \ref{theo_killer}}\label{sec:proof_killer}
The main objective of this section is to prove Theorems \ref{theo_bound} \& \ref{theo_killer}.   We will be needing the following lemma.
\begin{lemma}\label{c_pos}
Let $U$ denote an open subset of $(M, \omega)$.  If there exists a constant $E > 0$ such that  $c(H) \leq E$ for any Hamiltonian $H$ whose support is contained in $U$, then $0 \leq c(H)$ for any Hamiltonian supported in $U$.
\end{lemma}
\begin{proof}
Recall that the inverse flow $(\phi^t_H)^{-1}$ is generated by $\bar{H}(t,x) = -H(t, \phi^t_H(x))$.  Since $\bar{H} \# H = 0,$ using the triangle inequality we obtain  $-c(\bar{H}) \leq c(H).$  The Hamiltonian $\bar{H}$ is supported in $U$ and so  $c(\bar{H}) \leq E $.  Hence, it follows that  
\begin{equation}\label{lowerbound}
- E \leq c(H) \leq E
\end{equation}
For a contradiction, suppose that $c(H) < 0$.  By the triangle inequality, $c(H\# \cdots \# H) \leq m \, c(H)$, where $m$ denotes the total number of $\#'s$.  Taking $m$ to be sufficiently  large we find a Hamiltonian supported in $U$ with spectral invariant smaller than $- E$.  But this contradicts inequality \ref{lowerbound}.
\end{proof}

Note that a displaceable set $U \subset M$ would satisfy the conditions of the above lemma.
Next, we prove Theorem \ref{theo_bound} assuming Theorem \ref{theo_killer}.

\begin{proof}[Proof of Theorem  \ref{theo_bound}]
Let $H_i$ denote the restriction of $H$ to $U_i$ so that $H = \sum_{i=1}^{k}{H_i}.$  Recall that by the Energy-Capacity inequality $|c(H_i)| \leq E_i$, where $E_i$ denotes the displacement energy of $U_i$.  Now the $H_i's$ have disjoint supports and so $H_i \# H_j = H_i + H_j$.  It follows from the triangle inequality \ref{triangle} that $c(H) \leq \sum{c(H_i)} \leq \sum{E_i}$.  It follows from Lemma \ref{c_pos} that $0 \leq c(H)$.

It remains to show that $c(H) \leq \pi r^2$.  By Theorem \ref{theo_killer}, we can pick spectral killers $K_i$, supported in $U_i$, such that $c(H_i + K_i) = 0$.  
Note that the Hamiltonians $\{H_i, K_j\}$  have mutually disjoint supports.  Therefore, the $\#$ operation among these Hamiltonians is just simple addition of functions.  Hence,  
$$(H_1 + K_1) \# \cdots \# (H_k + K_k) = \sum{H_i+K_i} = H + \sum{K_i}.$$  
Using the triangle inequality, we obtain that $c( H +  \sum{K_i} ) \leq  \sum{c(H_i + K_i)} = 0$.  On the other hand, the Hamiltonian $H +  \sum{K_i}$ is supported in $U_1 \sqcup \cdots \sqcup U_k$ and we have already shown that any such Hamiltonian has non-negative spectral invariant.  Therefore, $c( H +  \sum{K_i} ) = 0.$

Now, by the Continuity property of spectral invariants $c(H) - c( H +  \sum{K_i} ) \leq \Vert  \sum{K_i} \Vert_{\infty}$ and thus $c(H) \leq \Vert  \sum{K_i} \Vert_{\infty}.$  Note that $\Vert  \sum{K_i} \Vert_{\infty} = \pi r^2:$ This is because $\Vert K_i\Vert_{\infty} = \pi r_i^2$ and the $K_i$'s have disjoint supports.  This completes our proof.
\end{proof}

We finish this section with a proof of Theorem \ref{theo_killer}.
\begin{proof}[Proof of Theorem \ref{theo_killer}]
Throughout this proof we identify $U$ with $B_r,$ the symplectic ball of radius $r$. Let $m$ denote a positive number which is larger than the maximum of $H + K_{\epsilon}$.
Define $F: M \rightarrow \mathbb{R}$ to be an autonomous Hamiltonian with the following properties:

\begin{enumerate}
\item $F|_{B_r}$ has the form $F(z) = f(\frac{|z|^2}{2})$, where $f:[0, \infty) \rightarrow \mathbb{R}$ and $f = 0$ on $[\frac{(r -  \epsilon)^2}{2}, \infty)$.  Thus, $F$ is radial and is supported in $B_{r - \epsilon}$.
\item $f = m$ on $[0, \frac{(r - 4 \epsilon)^2}{2})$.  Thus,  $F = m$ on  $B_{r - 4 \epsilon}$.
\item $f$ decreases linearly on $[\frac{(r - 4 \epsilon)^2}{2}, \frac{(r - 3 \epsilon)^2}{2})$.  
\item $f = -\pi r^2$  on $[\frac{(r - 3 \epsilon)^2}{2}, \frac{(r - 2 \epsilon)^2}{2}).$  Thus, $F(z) = -\pi r^2$ on $\{z : r - 3 \epsilon \leq|z| \leq r - 2 \epsilon \}.$
\item $f$ increases linearly on $[\frac{(r - 2 \epsilon)^2}{2}, \frac{(r -  \epsilon)^2}{2})$.  
\end{enumerate}

\begin{figure}
\centering
\includegraphics{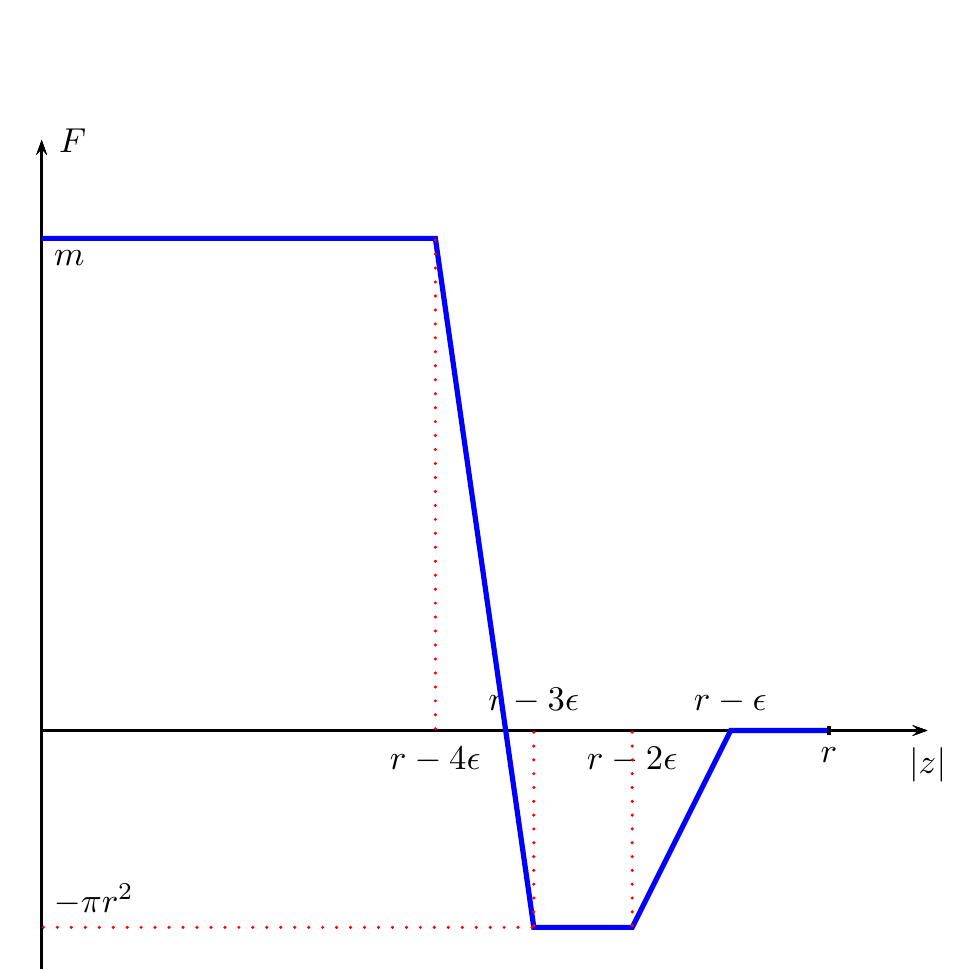}
\caption{Graph of the function F.}
\label{fig:F_graph}
\end{figure}

The graph of $F$ is depicted in Figure \ref{fig:F_graph}. To establish Theorem \ref{theo_killer} it is sufficient to prove that $c(F) = 0$:  Note that $H+ K_\epsilon \leq F$.  Hence, $c(H+ K_{\epsilon}) \leq 0$ by the Monotonicity property \ref{properties_spec}.    On the other hand, the set $U$ is displaceable and Lemma \ref{c_pos} implies that $c(H + K_\epsilon) \geq 0$. 

\medskip

The rest of this proof is dedicated to showing that $c(F) = 0$.  We will denote the displacement energy of $U$ by $E$.  It follows from the Energy-Capacity inequality appearing in Proposition 3.1 of \cite{usher} that $|c(F)| \leq E.$  Combining this with Lemma \ref{c_pos} we obtain
\begin{equation}\label{c_range}
c(F) \in [0, E].
\end{equation}
Furthermore, the Energy-Capacity inequality from Proposition 3.1 of \cite{usher} combined with the argument in Remark \ref{rem:upperbound_optimal}, implies that $\pi r^2 \leq E$, and hence we get
\begin{equation}\label{E_range}
\pi r^2 \leq E < \frac{|\lambda|}{2}.
\end{equation}

Here is an overview of our strategy for computing $c(F)$: We will perturb $F$, in several stages, to a smooth and non-degenerate Hamiltonian which we will continue to denote by $F$.  These perturbations will be performed in a fashion which ensures that all the non-trivial 1-periodic orbits of $F$  appear near $|z| \approx \; r - 4 \epsilon, \; r - 3 \epsilon, \; r - 2 \epsilon, \; r -  \epsilon.$  By Proposition \ref{properties_spec}, $c(F)$ will be attained by a capped 1-periodic orbit of $F$ with CZ index $n$.  We will show that the action of any such capped 1-periodic orbit is either approximately zero or falls outside of the half-open interval $(0, E]$.  Of course, in light of \eqref{c_range}, this forces $c(F)$ to be zero.  The fact that $F$ takes the value $-\pi r^2$ in the region $\{z : r - 3 \epsilon \leq|z| \leq r - 2 \epsilon \}$ is immensely important: Indeed, if the value of $F$ in this region was larger than $-\pi r^2$, then one would find a capped 1-periodic orbit near $|z| \approx  r - 3 \epsilon$ with CZ index $n$ and action inside the region $(0, E]$.  Furthermore, $c(F)$ would be attained by the action of this orbit, and so $c(F)$ would be larger than zero; see Remark \ref{rem:action_carrying_orbit} below.

\medskip

We begin the proof by replacing $f$ with a close by (in uniform norm) smooth function, which we continue to denote by $f$: Note that the non-trivial 1-periodic orbits of $F$  appear at values of $|z|$ such that $f'(\frac{|z|^2}{2}) = 2 \pi l$, where $l$ denotes a non-zero integer.  We take a smoothing of $f$ such that this occurs at $|z| \approx \; r - 4 \epsilon, \; r - 3 \epsilon, \; r - 2 \epsilon, \; r -  \epsilon.$  We leave it to the reader to verify that such smoothings do exist.  

At this point the Hamiltonian $F$ is smooth but degenerate.  Its 1-periodic orbits can be classified into seven sets: four sets of non-trivial 1-periodic orbits, as described in the previous paragraph, and three sets of trivial 1-periodic orbits corresponding to where the perturbed function is constant.  These trivial 1-periodic orbits appear inside the regions  $B_{r - \frac{7 \epsilon}{2}}, \; \{z : r - \frac{ 7 \epsilon} {2} < |z| < r - \frac{3\epsilon }{2} \},$ and $M \setminus B_{r- \frac{3\epsilon}{2}}$.  In what follows, we will  perturb $F$ to a non-degenerate Hamiltonian in seven steps.  At each step we analyze one of the above collections of 1-periodic orbits.  We first deal with the trivial 1-periodic orbits of $F$.\\

\noindent \textbf{Step 1: 1-periodic orbits in $B_{r -  \frac{7 \epsilon}{2}}$.} 

\noindent  In this region $F$ is constant and equal to $m$.  We make a $C^2$-small modification of $f$ on $[0, \frac{(r - 4 \epsilon)^2}{2})$ such that $f$ is left with a unique  critical point at zero and no other critical points in this interval;  the critical point at zero will be a maximum. After this modification, the (capped) 1-periodic orbits of $F$ in this region will be non-degenerate and of the form $[0, A]$ where $A \in \pi_2(M)$.  Of these, only $[0, u_0]$ has CZ-index $n$; here $u_0$ denotes the trivial capping disc.
The action of this orbit is $F(0) \approx m$.  By taking $m$ to be sufficiently large, we see, using Equation \eqref{c_range}, that $c(F)$ is not attained here.\\

\noindent \textbf{Step 2: 1-periodic orbits in $\{z : r - \frac{ 7 \epsilon} {2} < |z| < r - \frac{3\epsilon }{2} \}.$} 

\noindent  In this region $F$ is constant and equal to $-\pi r^2$.  After making a $C^2$-small autonomous perturbation, the restriction of $F$ to this region will have a finite number of non-degenerate critical points.  The corresponding capped 1-periodic orbits of $F$ in this region will be of the form $[p, A],$ where $p$ is a critical point of $F$ and $A \in \pi_2(M)$. The action of $[p, A]$ is approximately  $-\pi r^2 - \omega(A) = -\pi r^2 - \lambda c_1(A)$.  Inequality \eqref{E_range} implies that the smallest non-negative number of this form is larger than $E$.  Hence, we see, using Equation \eqref{c_range}, that $c(F)$ can not be attained by any of these orbits.\\

\noindent \textbf{Step 3: 1-periodic orbits in $M \setminus B_{r- \frac{3\epsilon}{2}}$.} 

\noindent  In this region, $F$ is constant and equal to $0$.  After making a $C^2$-small autonomous perturbation, the restriction of $F$ to this region will have a finite number of non-degenerate critical points.  The corresponding capped 1-periodic orbits of $F$ in this region will be of the form $[p, A],$  where $p$ is a critical point of $F$ and $A \in \pi_2(M)$. The action of $[p, A]$ is approximately  $ - \omega(A) =  - \lambda c_1(A)$.  Inequality \eqref{E_range} implies that the smallest positive number of this form is larger than $E$.  Hence, we see, using Equation \eqref{c_range}, that if $c(F)$ is attained by one of these orbits, then $c(F) \approx 0$.\\

Next, we analyze the non-trivial 1-periodic orbits of $F$.  As mentioned earlier, these orbits appear at values of $|z|$ such that $f'(\frac{|z|^2}{2}) = 2 \pi l$, where $l$ denotes a non-zero integer.  Let $z(t)$ be one such orbit and denote by $u$ a capping of $z(t)$ which is contained entirely inside $B_r$.  A simple computation would reveal that the action of this orbit is given by the expression: \begin{equation}\label{action} \mathcal{A}_F([z(t), u]) = f(\frac{|z|^2}{2}) - f'(\frac{|z|^2}{2}) \frac{|z|^2}{2}. \end{equation}

 Consider a value of $|z|$ where $f'(\frac{|z|^2}{2}) = 2 \pi l$.  At such a value of $|z|$ the periodic orbits form a $(2n-1)$-dimensional sphere; denote this sphere by $S_l$.  We will now describe our procedure for perturbing $F$ to a non-degenerate Hamiltonian:  pick a Morse function $h: S_l \rightarrow \mathbb{R}$ with exactly two critical points.  Identify $S_l \times (1- 2\varepsilon, 1+ 2\varepsilon)$ with a tubular neighborhood of $S_l$  via the map $(z, a) \mapsto az$, and extend $h$ to this neighborhood by setting $h(z,a) = h(z)g(a)$ where $g$ is any function with compact support in $(1- 2\varepsilon, 1+ 2\varepsilon)$ and such that $g|_{ (1- \varepsilon, 1+ \varepsilon)}  = 1$.   Consider the Hamiltonian $F + \delta h, \, \delta >0$.  For sufficiently small values of $\delta$, the Hamiltonian $F + \delta h$ will have exactly two 1-periodic orbits near $S_l$ corresponding to the two critical points of $h$.  We repeat the described perturbation near each sphere $S_l$.  In the end, we obtain a non-degenerate Hamiltonian, which we continue to denote by $F$, with two 1-periodic orbits near each value of $|z|$ such that $f'(\frac{|z|^2}{2}) = 2 \pi l$.  Since these perturbations are $C^2$-small, we can approximate the actions of these  orbits using Equation \eqref{action}.  For further details see \cite{CFHW, Oa}.  Here, we are following, very closely, Section 3.3 of \cite{Oa}.

  In \cite{Oa}, Oancea considers radial Hamiltonians of the form $f(\frac{|z|^2}{2})$ in $\mathbb{R}^{2n}$ and computes the Conley-Zehnder indices of the 1-periodic orbits which arise as a result of the perturbation performed in the previous paragraph.   Let $z^l_1(t), \; z^l_2(t)$ denote the two 1-periodic orbits of $F$ near $S_l$, corresponding respectively to the minimum and maximum of $h$. Extracting the indices of these orbits from Section 3.3 of \cite{Oa} we obtain:
	\begin{equation} \label{CZ_1}
 \mu_{CZ}([z^l_1(t), u] ) = \left\{ \begin{array}{ll}
         -2ln - n & \mbox{if $f''(\frac{|z|^2}{2}) > 0$ near $S_l$};\\
         -2ln -n + 1 & \mbox{if $ f''(\frac{|z|^2}{2}) < 0$ near $S_l$}.\end{array} \right. \end{equation}

\begin{equation} \label{CZ_2}
 \mu_{CZ}([z^l_2(t), u] ) = \left\{ \begin{array}{ll}
         -2ln + n -1  & \mbox{if $f''(\frac{|z|^2}{2}) > 0$ near $S_l$};\\
         -2ln + n  & \mbox{if $ f''(\frac{|z|^2}{2}) < 0$ near $S_l$}.\end{array} \right. \end{equation}
	
In the above, $u$ is the capping disc contained entirely in $B_r$.  Here, we will not perform the lengthy calculations needed to obtain these formulas.  Computing indices for radial Hamiltonians is folklore knowledge; we refer the interested reader to Section 3.3 of \cite{Oa}.  For comparing the above formulas with the formula presented towards the end of Section 3.3 of \cite{Oa}, we ask the reader to keep the following two points in mind:
\begin{itemize}
\item Oancea computes the negative of the Conley-Zehnder indices.  Hence, our formulas have the opposite sign.
\item In \cite{Oa} $f$ is convex, i.e. $f'' > 0$.  However, in our case $f'' < 0$ near $|z| \approx \; r - 4 \epsilon, \; r -  \epsilon$ and $f'' > 0$ near $|z| \approx \; r - 3 \epsilon, \; r - 2 \epsilon$.  To get the correct CZ indices in the cases where $f'' < 0$ one must add $1$ to the corresponding CZ indices from the cases where $f'' > 0$.
We leave it to the reader to check that this follows from the computations in \cite{Oa}.
\end{itemize}


The capped 1-periodic orbits of $F$  are of the form $[z^l_i(t), u \# A]$, where $z^l_i(t)$ are as described above, $u$ is a capping for $z^l_i(t)$ contained entirely in $B_r$, and $A \in \pi_2(M)$. Furthermore, recall that these orbits occur at  $|z| \approx \; r - 4 \epsilon, \; r - 3 \epsilon, \; r - 2 \epsilon, \; r -  \epsilon.$  We will now analyze the actions and indices of these orbits. 

\medskip

\noindent \textbf{Step 4: 1-periodic orbits near $|z| \approx \; r - 4 \epsilon$.} 

Using Equations \eqref{CZ_1}, \eqref{CZ_2}, and \eqref{CZ-index identity} we get the following values for the CZ indices of such orbits (note that here $f'' < 0$):
$$ \mu_{CZ}([z^l_i(t), u \# A]) = \left\{ \begin{array}{ll}
         -2ln - n  + 1 - 2c_1(A) & \mbox{$i=1$};\\
         -2ln + n - 2c_1(A) & \mbox{$i=2$}.\end{array} \right.$$
				
Recall that $c(F)$ is attained by a 1-periodic orbit of index $n$.  We see from the above formula that the capped orbit $[z^l_i(t), u \# A]$ has index $n$ only when $i=2$ and $c_1(A) = -ln$.
Using Equation \eqref{action}, we see that the action of this orbit is:
$$\mathcal{A}_F([z^l_2(t), u\#A]) \approx f(\frac{(r - 4 \epsilon)^2}{2}) - f'(\frac{(r - 4 \epsilon)^2}{2}) \frac{(r - 4 \epsilon)^2}{2} - \omega(A)$$
$$\approx m - l \pi (r - 4 \epsilon)^2 - \lambda c_1(A) = m - l \pi (r - 4 \epsilon)^2 - \lambda (-ln)$$ $$ = m + l (n \lambda - \pi (r - 4 \epsilon)^2 ).$$
  By increasing $m$, we may assume that it is an integer multiple of $(n \lambda - \pi (r - 4 \epsilon)^2)$.  Of course, this implies the actions of all such orbits are integer multiples of $(n \lambda - \pi (r - 4 \epsilon)^2)$.  Inequality \eqref{E_range} implies that the smallest positive number of this form (i.e. $|n \lambda - \pi (r - 4 \epsilon)^2|$) is larger than $E$.  Hence, we see, using Equation \eqref{c_range}, that if $c(F)$ is attained by one of these orbits, then $c(F) \approx 0$. 
	
\medskip

\noindent \textbf{Step 5: 1-periodic orbits near $|z| \approx \; r - 3 \epsilon$.} 

Using the same reasoning as in Step 4, we see that the CZ indices of these orbits are given by (note that here $f'' > 0$): 
$$ \mu_{CZ}([z^l_i(t), u \# A]) = \left\{ \begin{array}{ll}
         -2ln - n   - 2c_1(A) & \mbox{$i=1$};\\
         -2ln + n -1 - 2c_1(A) & \mbox{$i=2$}.\end{array} \right.$$
				
We see from the above formula that the capped orbit $[z^l_i(t), u \# A]$ has index $n$ only when $i=1$ and $c_1(A) = -n(l+1)$; the action of this orbit is: 
 $$\mathcal{A}_F([z^l_1(t), u\#A]) \approx f(\frac{(r - 3 \epsilon)^2}{2}) - f'(\frac{(r - 3 \epsilon)^2}{2}) \frac{(r - 3 \epsilon)^2}{2} - \omega(A)$$
$$\approx -\pi r^2 - l \pi (r - 3 \epsilon)^2 - \lambda c_1(A) = -\pi r^2 - l \pi (r - 3 \epsilon)^2 - \lambda (-n(l+1))$$
$$= (n \lambda - \pi r^2) + l (n \lambda - \pi (r - 3 \epsilon)^2 ).$$

Now, near $ \frac{(r - 3 \epsilon)^2}{2}$, $f$ is decreasing and hence $f'(\frac{(r - 3 \epsilon)^2}{2}) = 2 \pi l$ is negative.  Therefore, $l$ is negative.  

First, suppose that $ l \leq -2$.  In this case, one can check, using Inequality \eqref{E_range}, that the above action value is negative if $\lambda > 0$, and it is larger than $E$ if $\lambda < 0.$  Of course, Equation  \eqref{c_range} rules out such values for $c(F)$. 

 If $l= -1$, then the action value is $\pi (r - 3 \epsilon)^2 - \pi r^2$. This value is negative as well and so  $c(F)$ is not attained here either. 

\begin{remark} \label{rem:action_carrying_orbit}
It is precisely here, in the case $l=-1$ where we use the fact that the spectral killer $K_{\epsilon}$, and the function $F$, take the value $-\pi r^2$ on $\{z : r - 3 \epsilon \leq|z| \leq r - 2 \epsilon \}$. 

 If we were to modify $F$ such that it would take the value $a$ in this region, where $ -\pi r^2 < a \leq 0,$ then the orbit considered above in the case $l=-1$ would have action $\pi (r - 3 \epsilon)^2 + a,$ which falls in the range $(0,E)$, and hence it could be the case that $c(F)$ is attained by this \emph{non-zero} value. 
 
\end{remark}

\medskip
				
\noindent \textbf{Step 6: 1-periodic orbits near $|z| \approx \; r - 2 \epsilon$.} 

Using the same reasoning as in the previous step, we see that the CZ indices of these orbits are given by (note that here $f'' > 0$): 
$$ \mu_{CZ}([z^l_i(t), u \# A]) = \left\{ \begin{array}{ll}
         -2ln - n   - 2c_1(A) & \mbox{$i=1$};\\
         -2ln + n -1 - 2c_1(A) & \mbox{$i=2$}.\end{array} \right.$$
				
We see from the above formula that the capped orbit $[z^l_i(t), u \# A]$ has index $n$ only when $i=1$ and $c_1(A) = -n(l+1)$; the action of this orbit is: 
 $$\mathcal{A}_F([z^l_1(t), u\#A]) \approx f(\frac{(r - 2 \epsilon)^2}{2}) - f'(\frac{(r - 2 \epsilon)^2}{2}) \frac{(r - 2 \epsilon)^2}{2} - \omega(A)$$
$$\approx -\pi r^2 - l \pi (r - 2 \epsilon)^2 - \lambda c_1(A) = -\pi r^2 - l \pi (r - 2 \epsilon)^2 - \lambda (-n(l+1))$$
$$= (n \lambda - \pi r^2) + l (n \lambda - \pi (r - 2 \epsilon)^2 ).$$

Now, near $ \frac{(r - 2 \epsilon)^2}{2}$, $f$ is increasing and hence $f'(\frac{(r - 2 \epsilon)^2}{2}) = 2 \pi l$ is positive.  Therefore, $l$ is positive. 
One can check, using Inequality \eqref{E_range}, that the above action value is negative if $\lambda < 0$, and it is larger than $E$ if $\lambda > 0.$  Hence, we see, using Equation  \eqref{c_range}, that $c(F)$ is not attained by any of these orbits.

\medskip

\noindent \textbf{Step 7: 1-periodic orbits near $|z| \approx \; r -  \epsilon$.} 

Using the same reasoning as in the previous steps, we see that the CZ indices of these orbits are given by (note that here $f'' < 0$): 
$$ \mu_{CZ}([z^l_i(t), u \# A]) = \left\{ \begin{array}{ll}
         -2ln - n  + 1 - 2c_1(A) & \mbox{$i=1$};\\
         -2ln + n - 2c_1(A) & \mbox{$i=2$}.\end{array} \right.$$
				
We see from the above formula that the capped orbit $[z^l_i(t), u \# A]$ has index $n$ only when  $i=2$ and $c_1(A) = -ln$; the action of this orbit is: 

$$\mathcal{A}_F([z^l_2(t), u\#A]) \approx f(\frac{(r -  \epsilon)^2}{2}) - f'(\frac{(r -  \epsilon)^2}{2}) \frac{(r -  \epsilon)^2}{2} - \omega(A)$$
$$\approx  0 - l \pi (r -  \epsilon)^2 - \lambda (-ln)$$ $$ =   l (n \lambda - \pi (r -  \epsilon)^2 ).$$

Now, near $ \frac{(r -  \epsilon)^2}{2}$, $f$ is increasing and hence $f'(\frac{(r - \epsilon)^2}{2}) = 2 \pi l$ is positive.  Therefore, $l$ is positive. 
One can check, using Inequality \eqref{E_range}, that the above action value is negative if $\lambda < 0$, and it is larger than $E$ if $\lambda > 0.$  Hence, we see, using Equation  \eqref{c_range}, that $c(F)$ is not attained by any of these orbits.

\medskip

 In summary, through steps 1-7, we have shown that the actions of 1-periodic orbits of $F$ with Conley-Zehnder index $n$ are  either approximately zero or fall outside the interval $(0, E]$.  Hence, we conclude that $c(F) = 0$.
\end{proof}

\section{Proof of Theorem \ref{theo_1}}\label{sec:proof_theo_1}
As mentioned in the introduction, our main motivation for seeking results in the spirit of Theorem \ref{theo_bound} is the application of such results to the theory of Poisson bracket invariants.  In this section we will explain how Theorem \ref{theo_1} follows from Theorem \ref{theo_bound}.   This proof is not due to the author; it was explained to us by Leonid Polterovich.  This proof can also be extracted from Section 5 of his article \cite{Po12b}. 

\subsection*{A partial symplectic quasi-state}  We begin by introducing the  \emph{partial symplectic quasi-state} $\zeta: C^{\infty}(M) \rightarrow \mathbb{R}$, defined as follows:
$$\zeta(F) = \lim_{k \to\infty} \frac{c(kF)}{k},$$
where $c$ is the spectral invariant defined in Section \ref{sec:sepctral_invariants}. The functional $\zeta$ was introduced by Entov and Polterovich in \cite{EP06}.  We will need the following two properties of $\zeta$:

\begin{enumerate}
   \item (Monotonicity) If $H \leq G$ then $\zeta(H) \leq \zeta(G)$,
   \item (Normalization) $\zeta(C) =C$, for any constant  $C \in \mathbb{R}$. 
\end{enumerate}
The above two properties can be deduced from Proposition \ref{properties_spec}.

\subsection*{The Poisson bracket inequality} For $F, G \in C^{\infty}(M)$ define $$\Pi(F,G) =| \zeta(F+G) - \zeta(F) -\zeta(G)|$$ and $$S(F,G) = \sup_{s>0} \min\{c(sF) + c(-sF), c(sG) + c(-sG)\}.$$
The Poisson bracket inequality states that $$\Pi(F,G) \leq \sqrt{2S(F,G) \Vert \{F,G\} \Vert_{\infty}}.$$
For a proof see Proposition 5.2 of \cite{Po12b}.  A slightly different version of this inequality was introduced and proven in \cite{EPZ07}.

\begin{proof}[Proof of Theorem \ref{theo_1}]
  Because the covering $\mathcal{U} = \{U_1, \cdots, U_N \}$ is $d$-regular, it can be partitioned into (at most) $d+1$ subsets, say $\mathcal{W}_1, \cdots, \mathcal{W}_{d+1} \subset \mathcal{U}$, such that each collection $\mathcal{W}_j$ consists of mutually disjoint sets. This is proven in Section 4.5 of \cite{Po12b}; see Proposition 4.7. 

For $ 1 \leq j \leq d+1$ let $\displaystyle W_j = \cup_{U_i \in \mathcal{W}_j} U_i$.  Consider a partition of unity $\vec{f} = \{f_1, \cdots, f_L\}$ subordinate to $\mathcal{U}.$  For $ 1 \leq j \leq d+1$ let $$F_j = \sum_{Supp(f_i) \subset W_j}{f_i}.$$ 

Write $G_k = F_1 + \cdots + F_k$, for $k = 1, \cdots, d+1.$  Applying the Poisson bracket inequality to $G_{k+1}$ and $G_k$ we get:
\begin{equation} \label{PB} |\zeta(G_{k+1}) - \zeta(G_k) - \zeta(F_{k+1})| \leq \sqrt{ 2S(G_k, F_{k+1}) \Vert \{G_{k}, F_{k+1} \} \Vert}.\end{equation}

We now analyze and simplify the above inequality.  

First, notice that \begin{equation} \label{lowerbound_nu} \Vert \{G_{k}, F_{k+1} \} \Vert_{\infty} \leq \nu_c(\vec{f}). \end{equation}

Second, we claim that Theorem \ref{theo_bound} implies that \begin{equation} \label{upperbound_S} S(G_k, F_{k+1}) \leq \pi r^2. \end{equation} Indeed, it follows directly from Theorem $\ref{theo_bound}$ that $c(sF_j) \leq \pi r^2$ for all $s >0$ and $1 \leq j \leq d+1$.  On the other hand, $-sF_j \leq 0$ and so $c(-sF_j) \leq 0$.  Hence, $c(sF_j) + c(-sF_j) \leq \pi r^2$ for all $s >0$ and $1 \leq j \leq d+1$.  

Third, the fact that $c(sF_j) \leq \pi r^2$ for all $s >0$ and $1 \leq j \leq d+1$ implies that \begin{equation} \label{zeta_F} \zeta(F_{k+1}) =0. \end{equation}

Using Equations \eqref{lowerbound_nu}, \eqref{upperbound_S}, and \eqref{zeta_F} we simplify Inequality \eqref{PB} and get
$$\zeta(G_{k+1}) \leq \zeta(G_k) + \sqrt{2 \pi r^2 \nu_c(\vec{f})}.$$
Thus, $1 = \zeta(G_{d+1}) \leq d \sqrt{2 \pi r^2 \nu_c(\vec{f})},$ and hence $\frac{1}{2 d^2 \pi r^2} \leq \nu_c(\vec{f}).$
This proves Theorem \ref{theo_1}.
\end{proof}

\bibliographystyle{abbrv}
\bibliography{biblio}

\end{document}